\documentclass[11pt]{amsart}
\usepackage[T1]{fontenc}
\usepackage[utf8]{inputenc}
\usepackage{lmodern,amsmath,amssymb,amsthm,mathtools,booktabs,microtype}
\usepackage[margin=1.05in]{geometry}
\usepackage{xcolor}
\usepackage[colorlinks=true,linkcolor=blue!50!black,citecolor=blue!50!black,urlcolor=blue!50!black]{hyperref}
\hypersetup{pdftitle={Generation by lazy and invertible cellular automata over the Klein four-group},pdfauthor={Edgar Alcala-Arroyo, Alonso Castillo-Ramirez, Haobo Ma, Wenlin Zhang}}
\newtheorem{theorem}{Theorem}[section]
\newtheorem{proposition}[theorem]{Proposition}
\newtheorem{lemma}[theorem]{Lemma}
\newtheorem{corollary}[theorem]{Corollary}
\theoremstyle{definition}
\newtheorem{definition}[theorem]{Definition}
\newtheorem{example}[theorem]{Example}
\newtheorem{remark}[theorem]{Remark}
\DeclareMathOperator{\CA}{CA}
\DeclareMathOperator{\ICA}{ICA}
\DeclareMathOperator{\End}{End}
\DeclareMathOperator{\Aut}{Aut}
\DeclareMathOperator{\supp}{supp}

\newcommand{\id}{\mathrm{id}}
\newcommand{\one}{\mathbf{1}}
\newcommand{\cC}{\mathcal C}
\newcommand{\cF}{\mathcal F}
\newcommand{\cK}{\mathcal K}
\newcommand{\cN}{\mathcal N}
\newcommand{\cU}{\mathcal U}
\newcommand{\cL}{\mathcal L}
\newcommand{\cM}{\mathcal M}
\newcommand{\cD}{\mathcal D}
\newcommand{\cE}{\mathcal E}
\newcommand{\cQ}{\mathcal Q}
\newcommand{\eps}{\varepsilon}

\title[Generation over the Klein four-group]{Generation by lazy and invertible cellular automata over the Klein four-group}
\author{Edgar Alcal\'{a}-Arroyo}
\address[E. Alcal\'{a}-Arroyo]{Universidad de Guadalajara, Centro Universitario de Ciencias Exactas e Ingenier\'{i}as, Departamento de Matem\'{a}ticas. Guadalajara, M\'{e}xico}
\email{edgar.alcala7434@alumnos.udg.mx}
\author{Alonso Castillo-Ramirez}
\address[A. Castillo-Ramirez]{Universidad de Guadalajara, Centro Universitario de Ciencias Exactas e Ingenier\'{i}as, Departamento de Matem\'{a}ticas. Guadalajara, M\'{e}xico}
\email{alonso.castillor@academicos.udg.mx}
\author{Haobo Ma}
\address[H. Ma]{ChronoAI Pte Ltd; The Omega Institute}
\email{auric@aelf.io}
\author{Wenlin Zhang}
\address[W. Zhang]{National University of Singapore (NUS); The Omega Institute}
\email{e1327962@u.nus.edu}
\date{\today}
\begin{document}
\begin{abstract}
Let $G=C_2\times C_2$ and $A=\{0,1\}$. We characterize the submonoid of $\CA(G;A)$ generated by lazy and invertible cellular automata. A cellular automaton belongs to this submonoid if and only if, whenever it maps the two free $G$-orbits onto the same free orbit, it maps both constant configurations and some two-point orbit to the same constant. The proof gives a factorization using elementary orbit collapses and three two-cell lazy maps. The criterion yields the cardinality $59\,712$ and identifies the maps missed by a weaker constant-separation invariant. Adjoining one idempotent generates the full monoid, and the generated submonoid is contained in exactly one maximal proper submonoid, of cardinality $65\,024$. The submonoid generated by lazy maps alone has $4\,007$ elements, and for both generating sets the maximum shortest-word length is six. The criterion and the cardinalities $59\,712$ and $4\,007$ are formally verified in Lean.
\end{abstract}
\maketitle

\section{Introduction}

A lazy cellular automaton acts as the identity except at the cells where one prescribed pattern occurs, where it changes the central symbol. Alcal\'{a}-Arroyo and Castillo-Ramirez~\cite[Problem~2]{order} asked whether lazy and invertible cellular automata generate the full cellular-automaton monoid. The binary full shift over the Klein four-group provides an obstruction, as shown in~\cite{generation}, which studies equality with the full monoid for general groups and alphabets, including a classification when both are finite. Here we describe the proper generated submonoid in the binary Klein-four case, including its elements, cardinality, and word lengths.

Put $G=C_2\times C_2$ and $X=\{0,1\}^G$. As $G$ is finite, the monoid $\CA(G;\{0,1\})$ of all cellular automata over $X$ coincides with the endomorphism monoid of the $G$-set $X$ with respect to the shift action~\cite{finite}: that is, the set of all $G$-equivariant transformations $T : X \to X$ equipped with composition. Write $\cL$ for the set of lazy cellular automata over $X$, $\cU=\ICA(G;\{0,1\})$ for the group of invertible cellular automata over $X$, and
\[
\cM=\langle\cU\cup\cL\rangle.
\]
Generated submonoids include the identity, represented by the empty word. The action of $G$ on $X$ has two fixed points, three two-point orbits, and two free orbits. The free orbits consist of the configurations of weights one and three, respectively.

\begin{theorem}\label{thm:main}
Let $T\in\CA(G;\{0,1\})$. Then $T\in\cM$ if and only if the following condition holds: if $T$ maps the two free orbits onto the same free orbit, then it maps both constant configurations and some two-point orbit to the same constant. Moreover,
\[
|\cM|=59\,712.
\]
\end{theorem}

Formally, the condition in Theorem~\ref{thm:main} is~\eqref{eq:criterion} in Section~\ref{sec:criterion}. In particular, $\cM\subsetneq\CA(G;\{0,1\})$, giving a negative answer to~\cite[Problem~2]{order}. The criterion and the cardinalities of $\cM$ and $\langle\cL\rangle$ are also formally verified in Lean~4~\cite{lean4}, using Mathlib~\cite{mathlib}; see the supplementary material~\cite{supp}.

We also show that one additional idempotent generates the full monoid and that $\cM$ lies in a unique maximal proper submonoid; see Proposition~\ref{prop:relative-rank} and Theorem~\ref{thm:maximal}.

The proof uses elementary collapses between $G$-orbits. This approach is related to generation of endomorphism monoids of finite $G$-sets~\cite{rank} and to the structure of cellular-automaton monoids over finite groups~\cite{finite}. We give the required constructions explicitly. In particular, the proof of Theorem~\ref{thm:main} does not depend on exhaustive enumeration.

Section~\ref{sec:orbits} describes the orbits and the elementary collapses. Section~\ref{sec:criterion} proves the criterion and a factorization result. Section~\ref{sec:count} gives the count and compares the criterion with a weaker constant-separation invariant. Section~\ref{sec:maximal} establishes the relative rank and describes the maximal proper submonoid containing $\cM$. Section~\ref{sec:computation} records the exact computation of $\langle\cL\rangle$, the word lengths in both submonoids, and the formal verification in Lean.

\section{Orbits and elementary collapses}\label{sec:orbits}

Throughout, let
\[
G=\{e,a,b,ab\},\qquad a^2=b^2=e,\qquad ab=ba,
\]
and let $X=\{0,1\}^G$. We use the action
\[
(g\cdot x)(h)=x(hg)\qquad(g,h\in G,\ x\in X).
\]
For $x\in X$, write $Gx=\{g\cdot x:g\in G\}$ and $G_x=\{g\in G:g\cdot x=x\}$. Define $w(x)=|\supp(x)|$, where $\supp(x)=\{g\in G:x(g)=1\}$. For $D\subseteq G$, let $\one_D$ be its characteristic configuration, and put $\delta_g=\one_{\{g\}}$. Write $c_i$ for the constant configuration with value $i$, and $\bar x=1-x$.

A map $T:X\to X$ is a \emph{cellular automaton} if there exist $S\subseteq G$ and $\mu:\{0,1\}^S\to\{0,1\}$ such that
\[
T(x)(g)=\mu((g\cdot x)|_S).
\]
Since $G$ is finite, these are precisely the $G$-equivariant transformations: the converse follows by taking $S=G$ and $\mu(x)=T(x)(e)$. Thus
\[
\CA(G;\{0,1\})=\End_G(X),\qquad \cU=\Aut_G(X).
\]
We compose maps from right to left. Equivariance implies
\begin{equation}\label{eq:stab}
G_x\leq G_{T(x)}\qquad(x\in X).
\end{equation}

\begin{definition}
For $e\in S\subseteq G$ and $p\in\{0,1\}^S$, the \emph{lazy cellular automaton} $\lambda_{S,p}$ is defined by
\[
\lambda_{S,p}(x)(g)=
\begin{cases}
1-x(g),&(g\cdot x)|_S=p,\\
x(g),&(g\cdot x)|_S\ne p.
\end{cases}
\]
The set $S$ is its active neighborhood and $p$ is its active pattern. Put $\cL_G=\{\lambda_{G,p}:p\in X\}$.
\end{definition}

For $s\in G\setminus\{e\}$, put $H_s=\{e,s\}$ and
\[
P_s=G\one_{H_s},\qquad O_i=\{x\in X:w(x)=i\}\quad(i=1,3).
\]
Also write $\cC=\{c_0,c_1\}$, $\cF=O_1\cup O_3$, and $Y=X\setminus\cF$.

\begin{lemma}\label{lem:orbits}
The orbit decomposition of $X$ is
\[
X=\{c_0\}\sqcup\{c_1\}\sqcup P_a\sqcup P_b\sqcup P_{ab}\sqcup O_1\sqcup O_3.
\]
The stabilizers on these orbits are $G,G,H_a,H_b,H_{ab},\{e\},\{e\}$, respectively. Consequently, for every $T\in\End_G(X)$,
\begin{equation}\label{eq:lower}
T(\cC)\subseteq\cC,\qquad T(P_s)\subseteq P_s\cup\cC,\qquad T(Y)\subseteq Y.
\end{equation}
Moreover,
\[
|\End_G(X)|=2^{16}=65\,536,\qquad |\cU|=2!\,2^3\,(2!\,4^2)=512.
\]
\end{lemma}
\begin{proof}
If $D=\{u,v\}\subseteq G$ has two elements, then $D$ is a coset of $H_{u^{-1}v}$. Hence every weight-two configuration belongs to exactly one $P_s$, and its stabilizer is $H_s$. A singleton and its complement have trivial stabilizer. This proves the decomposition, and~\eqref{eq:lower} follows from~\eqref{eq:stab}.

Every function $X\to\{0,1\}$ is the local map at $e$ of exactly one cellular automaton, giving $2^{16}$ maps. The inverse of an equivariant bijection is equivariant, so a unit preserves stabilizers exactly and hence preserves each $P_s$ setwise. A unit permutes the two constants, acts independently on each $P_s$ in two ways, and permutes the two free orbits, with four equivariant bijections on each. These choices give the stated count.
\end{proof}

\begin{definition}
If $Gx\ne Gy$ and $G_x\leq G_y$, define the \emph{elementary collapse}
\[
\eps_{x,y}(z)=
\begin{cases}
g\cdot y,&z=g\cdot x,\\
z,&z\notin Gx.
\end{cases}
\]
The stabilizer inclusion makes this map well defined. It is equivariant and idempotent.
\end{definition}

\begin{lemma}\label{lem:elementary}
Every elementary collapse of one of the following types belongs to $\langle\cU\cup\cL_G\rangle$:
\[
\{c_i\}\longrightarrow\{c_{1-i}\},\qquad
P_s\longrightarrow\{c_i\},\qquad
O_j\longrightarrow\{c_i\},\qquad
O_j\longrightarrow P_s,
\]
where $i\in\{0,1\}$, $j\in\{1,3\}$, and $s\ne e$.
\end{lemma}
\begin{proof}
For $p\in X$, the map $\lambda_{G,p}$ fixes $X\setminus Gp$ and changes $p$ exactly at the coordinates in $G_p$. Hence the following lazy maps realize one example of each type:
\[
\begin{array}{c|c}
\text{Lazy map}&\text{Elementary collapse}\\ \hline
\lambda_{G,c_0}&\eps_{c_0,c_1}\\
\lambda_{G,\one_{H_s}}&\eps_{\one_{H_s},c_0}\\
\lambda_{G,\delta_e}&\eps_{\delta_e,c_0}\\
\lambda_{G,\delta_s}&\eps_{\delta_s,\one_{H_s}}.
\end{array}
\]
For $\alpha\in\cU$,
\[
\alpha\eps_{x,y}\alpha^{-1}=\eps_{\alpha(x),\alpha(y)}.
\]
A unit can be chosen independently on $\cC$, on each $P_s$, and on $\cF$. The units that preserve every orbit form a subgroup acting transitively on each orbit, and other units interchange $O_1$ and $O_3$ or interchange the two constants. Thus every collapse in the statement is conjugate to one of the displayed maps.
\end{proof}

Say that $T$ \emph{merges the free orbits} if
\begin{equation}\label{eq:merge}
T(O_1)=T(O_3)\in\{O_1,O_3\}.
\end{equation}
This terminology always requires the common image to be free. Put
\[
\cN=\{T\in\End_G(X):T\text{ does not satisfy }\eqref{eq:merge}\}.
\]

\begin{proposition}\label{prop:N}
We have $\cN=\langle\cU\cup\cL_G\rangle$.
\end{proposition}
\begin{proof}
For $\alpha,\beta\in\End_G(X)$, suppose that $\beta\alpha$ merges the free orbits. For $x\in\cF$, applying~\eqref{eq:stab} to $\beta$ gives $G_{\alpha(x)}\leq G_{\beta\alpha(x)}=\{e\}$, so $\alpha(O_1)$ and $\alpha(O_3)$ are free. Either $\alpha$ merges them, or $\alpha$ permutes $\cF$ and $\beta$ merges them. Hence $\cN$ is a submonoid. Units belong to $\cN$. A full-neighborhood lazy map changes only one orbit: a changed free configuration has even weight, and a changed nonfree configuration remains nonfree. Thus $\cL_G\subseteq\cN$.

For the reverse inclusion, let $T\in\cN$. Its restrictions to the free orbits whose images are free form an injective equivariant partial permutation of $\cF$. Extend it to a unit $\alpha$ of $X$ fixing $Y$ pointwise, and put $T'=T\alpha^{-1}$. Every free orbit that $T'$ maps to a free orbit is fixed pointwise. All other free orbits map into $Y$.

The map $R$ agreeing with $T'$ on $Y$ and with the identity on $\cF$ is generated by units and the collapses in Lemma~\ref{lem:elementary}. Explicitly, first realize $T'|_{\cC}$ by a permutation or a collapse of the two constants. Then, by successive left compositions, realize on each $P_s$ either its prescribed permutation or its collapse to the prescribed constant. Each of these later maps fixes $X\setminus P_s$, so it does not move the constants already placed.

For every remaining free orbit $Q$, choose $x_Q\in Q$ and let $E_Q=\eps_{x_Q,T'(x_Q)}$. Its target lies in $Y$, so Lemma~\ref{lem:elementary} applies. The $E_Q$ commute, fix $Y$, and satisfy
\[
T'=\Bigl(\prod_Q E_Q\Bigr)R.
\]
Therefore $T=T'\alpha\in\langle\cU\cup\cL_G\rangle$.
\end{proof}

\section{The membership criterion}\label{sec:criterion}

For $s\ne e$, define
\begin{equation}\label{eq:B}
B_s=\lambda_{H_s,1^{H_s}},\qquad B_s(x)(g)=x(g)(1-x(gs)),
\end{equation}
where $1^{H_s}:H_s\to\{0,1\}$ is the constant pattern with value one. Let $J(x)=\bar x$; then $J\in\cU$ and $J^2=\id$.

\begin{lemma}\label{lem:B}
The map $B_s$ fixes $O_1$ pointwise, maps $O_3$ bijectively onto $O_1$, sends $\cC\cup P_s$ to $c_0$, and fixes $P_t$ pointwise for $t\ne s$. The lazy maps that merge the free orbits are exactly
\[
\{B_s,JB_sJ:s\in G\setminus\{e\}\}.
\]
\end{lemma}
\begin{proof}
Formula~\eqref{eq:B} fixes every singleton configuration. For a configuration with a unique zero at $u$, it leaves a unique one at $us$. It vanishes on configurations fixed by $s$. If $x\in P_t$ with $t\ne s$, then $x(gs)=1-x(g)$, so $B_s(x)=x$. These observations prove the first assertion.

Let $\lambda=\lambda_{S,p}$ merge the free orbits. If $p$ is nonconstant, it occurs at most once in any odd-weight configuration. Indeed, fix $h\in S$ with $p(h)=1$; an occurrence at $g$ in $\delta_u$ requires $\delta_u(hg)=1$, that is, $hg=u$, which determines $g$. For a weight-three configuration with its unique zero at $u$, argue in the same way with $h\in S$ such that $p(h)=0$. Therefore $\lambda$ either fixes such a configuration or changes one coordinate. An odd-weight image must then be unchanged, which precludes a merger.

Suppose $p$ is constant. Conjugating by $J$, we may assume $p=0^S$. A weight-three configuration remains of odd weight only if its unique zero does not change; thus $\lambda$ fixes $O_3$. If $x$ has a unique one at $u$, the active cells are precisely $G\setminus S^{-1}u$, so
\[
w(\lambda(x))=1+4-|S|.
\]
For $O_1$ to map onto $O_3$, we must have $|S|=2$. Hence $S=\{e,s\}$ and $\lambda=JB_sJ$. Conjugating back gives the stated list.
\end{proof}

Let $\cK$ be the set of $T\in\End_G(X)$ satisfying
\begin{equation}\label{eq:criterion}
\text{if $T$ merges the free orbits, then }
T(\cC\cup P_s)=\{c_i\}\text{ for some }s\ne e,\ i\in\{0,1\}.
\end{equation}

\begin{lemma}\label{lem:closed}
The set $\cK$ is a submonoid containing $\cU\cup\cL$.
\end{lemma}
\begin{proof}
Units satisfy~\eqref{eq:criterion}, and Lemma~\ref{lem:B} proves it for lazy maps. Let $\alpha,\beta\in\cK$ and suppose $\beta\alpha$ merges the free orbits. Then $\alpha$ maps each free orbit to a free orbit by~\eqref{eq:stab}.

If $\alpha$ merges them, choose $s,i$ with $\alpha(\cC\cup P_s)=\{c_i\}$. The composite maps this set to the constant $\beta(c_i)$. Otherwise, $\alpha$ permutes $\cF$, so $\beta$ merges the free orbits. Choose $s,i$ with $\beta(\cC\cup P_s)=\{c_i\}$. By~\eqref{eq:lower}, $\alpha(\cC\cup P_s)\subseteq\cC\cup P_s$, and hence $\beta\alpha(\cC\cup P_s)=\{c_i\}$. Thus $\beta\alpha\in\cK$.
\end{proof}

\begin{proposition}\label{prop:factor}
Let $T\in\cK$ merge the free orbits, and let $s\ne e$ and $i\in\{0,1\}$ satisfy $T(\cC\cup P_s)=\{c_i\}$. Then
\[
T=R B_s \alpha
\]
for some $R\in\cN$ and $\alpha\in\cU$ such that $\alpha|_Y=\id_Y$ and $R|_Y=T|_Y$.
\end{proposition}
\begin{proof}
Put $Q=T(O_1)=T(O_3)$. Choose an equivariant bijection $r:O_1\to Q$, and extend it to an equivariant permutation of $\cF$. Define $R$ by this permutation on $\cF$ and by $T$ on $Y$. Then $R\in\cN$.

For $j\in\{1,3\}$, the restriction $b_j=B_s|_{O_j}:O_j\to O_1$ is an equivariant bijection. Define
\[
\alpha|_{O_j}=b_j^{-1}\circ r^{-1}\circ T|_{O_j},\qquad \alpha|_Y=\id_Y.
\]
These restrictions are equivariant bijections, so $\alpha\in\cU$, and $RB_s\alpha=T$ on $\cF$. On $\cC\cup P_s$, both maps have value $c_i$, since $R(c_0)=T(c_0)=c_i$. On each $P_t$ with $t\ne s$, both $\alpha$ and $B_s$ are the identity and $R=T$. The factorization follows.
\end{proof}

\begin{proof}[Proof of the criterion in Theorem~\ref{thm:main}]
Lemma~\ref{lem:closed} gives $\cM\subseteq\cK$. Conversely, if $T\in\cK$ does not merge the free orbits, Proposition~\ref{prop:N} gives $T\in\cM$. Otherwise, apply Proposition~\ref{prop:factor}; its factors belong to $\cM$ by Proposition~\ref{prop:N} and~\eqref{eq:B}. Hence $\cM=\cK$.
\end{proof}

\begin{corollary}\label{cor:generators}
The following generating set and factorization description hold:
\[
\cM=\langle\cU\cup\cL_G\cup\{B_a,B_b,B_{ab}\}\rangle,
\qquad
\cM=\cN\ \cup\!\!\bigcup_{s\ne e}\cN B_s\cU.
\]
\end{corollary}
\begin{proof}
Every displayed factor belongs to $\cM$. The reverse inclusions follow from Propositions~\ref{prop:N} and~\ref{prop:factor}.
\end{proof}

\section{Counting and the weaker invariant}\label{sec:count}

An equivariant map on a free orbit is determined by the image of one representative, which may be any element of $X$. An equivariant map on $P_s$ is determined by the image of one representative, which may be either element of $P_s$ or either constant. These choices are independent on distinct domain orbits.

\begin{proposition}\label{prop:count}
There are $5\,824$ maps outside $\cM$, and
\[
|\cN|=57\,344,\qquad |\cM\setminus\cN|=2\,368,\qquad |\cM|=59\,712.
\]
\end{proposition}
\begin{proof}
There are $2\cdot4^2=32$ ways to map both free orbits onto the same free orbit: two choices for the target orbit and four equivariant bijections from each source orbit. By Theorem~\ref{thm:main}, such a map fails membership exactly in either of the following disjoint cases.

If the constants have distinct images, there are two choices for their restriction and four choices on each $P_s$. This gives $32\cdot2\cdot4^3=4\,096$ maps.

If both constants have image $c_i$, but no $P_s$ maps to $c_i$, there are two choices for $i$ and three choices on each $P_s$: its two automorphisms or its collapse to $c_{1-i}$. This gives $32\cdot2\cdot3^3=1\,728$ maps. Therefore
\[
|\cM|=65\,536-4\,096-1\,728=59\,712.
\]
There are $32\cdot4\cdot4^3=8\,192$ maps that merge the free orbits. Subtracting this from $65\,536$ gives $|\cN|=57\,344$, and the remaining equality follows.
\end{proof}

A weaker condition is the constant-separation invariant
\begin{equation}\label{eq:old}
T(c_0)\ne T(c_1)\ \Longrightarrow\ T|_{\cF\cap T^{-1}(\cF)}\text{ is injective}.
\end{equation}
An equivariant map between free orbits is bijective, so noninjectivity over a free image occurs exactly when the two free orbits merge. Thus~\eqref{eq:old} excludes the first family in Proposition~\ref{prop:count}, but excludes no map whose constants have the same image. It holds for $65\,536-4\,096=61\,440$ maps, of which exactly $1\,728$ lie outside $\cM$.

\begin{example}\label{ex:missed}
Define
\[
T(x)=\begin{cases}
\bar x,&w(x)=3,\\
c_0,&x=c_1,\\
x,&\text{otherwise}.
\end{cases}
\]
This map is equivariant, sends both constants to $c_0$, fixes every $P_s$, and maps both free orbits onto $O_1$. Hence it satisfies~\eqref{eq:old}, but $T\notin\cM$ by Theorem~\ref{thm:main}.
\end{example}

\section{Relative rank and maximality}\label{sec:maximal}

Put $\cE=\End_G(X)$ and $\operatorname{rk}(T)=|T(X)|$ for $T\in\cE$. The \emph{relative rank} of $\cM$ in $\cE$ is the least cardinality of a set $B\subseteq\cE$ such that $\cE=\langle\cM\cup B\rangle$. Define
\[
\eps(x)=
\begin{cases}
\bar x,&x\in O_3,\\
x,&x\notin O_3.
\end{cases}
\]
This is the elementary collapse $\eps_{\bar\delta_e,\delta_e}$. In particular, $\eps^2=\eps$ and $\operatorname{rk}(\eps)=12$.

\begin{proposition}\label{prop:relative-rank}
The relative rank of $\cM$ in $\cE$ is one. More precisely,
\[
\cE=\cN\cup\cN\eps\cU
=\langle\cM\cup\{\eps\}\rangle.
\]
\end{proposition}
\begin{proof}
Let $T\in\cE\setminus\cN$. Then $Q=T(O_1)=T(O_3)$ is free, and the restrictions $t_j=T|_{O_j}:O_j\to Q$ are equivariant bijections. Extend $t_1$ to an equivariant permutation of $\cF$, and define $R$ by this permutation on $\cF$ and by $T$ on $Y$. Thus $R\in\cN$ and $R|_Y=T|_Y$.

Put $\phi=\eps|_{O_3}:O_3\to O_1$, and define $\alpha\in\cU$ by
\[
\alpha|_{Y\cup O_1}=\id_{Y\cup O_1},\qquad
\alpha|_{O_3}=\phi^{-1}\circ t_1^{-1}\circ t_3.
\]
Then $R\eps\alpha=T$: on $O_3$ this follows from $t_1\phi\phi^{-1}t_1^{-1}t_3=t_3$, and on $Y\cup O_1$ it follows from the definitions. Hence $\cE=\cN\cup\cN\eps\cU$. Since $\cN\subseteq\cM$ by Proposition~\ref{prop:N}, adjoining $\eps$ to $\cM$ generates $\cE$. Finally, $\eps\notin\cM$ by Theorem~\ref{thm:main}, since it merges the free orbits and fixes both constants. Thus the relative rank is one.
\end{proof}

\begin{theorem}\label{thm:maximal}
Let
\[
I_{11}=\{T\in\cE:\operatorname{rk}(T)\leq11\},\qquad
\cQ=\cM\cup I_{11}.
\]
Then $\cQ$ is the unique maximal proper submonoid of $\cE$ containing $\cM$. Moreover,
\[
\cM\subsetneq\cQ\subsetneq\cE,\qquad
\cE\setminus\cQ=\cU\eps\cU,\qquad
|\cQ|=65\,024.
\]
\end{theorem}
\begin{proof}
Since $\operatorname{rk}(AB)\leq\min\{\operatorname{rk}(A),\operatorname{rk}(B)\}$, the set $I_{11}$ is a two-sided ideal of $\cE$. Thus $\cQ$ is a submonoid.

Every $T\notin\cM$ merges the free orbits. Since $T(Y)\subseteq Y$, we have
\[
\operatorname{rk}(T)=4+|T(Y)|\leq12,
\]
with equality precisely when $T|_Y$ is bijective. In this case the factor $R$ constructed in Proposition~\ref{prop:relative-rank} is a unit, so $T\in\cU\eps\cU$. Conversely, every $u\eps v$ with $u,v\in\cU$ has rank $12$ and lies outside $\cM$; otherwise $\eps=u^{-1}(u\eps v)v^{-1}$ would belong to $\cM$. Therefore $\cE\setminus\cQ=\cU\eps\cU$.

If a submonoid $H$ contains $\cM$ and some $T=u\eps v\notin\cQ$, then $\eps=u^{-1}Tv^{-1}\in H$. Proposition~\ref{prop:relative-rank} gives $H=\cE$. Consequently, every proper submonoid containing $\cM$ is contained in $\cQ$. Since $\eps\notin\cQ$, this proves that $\cQ$ is the unique maximal proper submonoid containing $\cM$.

There are $2\cdot4^2=32$ restrictions merging the free orbits and $2\cdot2^3=16$ equivariant bijections of $Y$. Every combination lies outside $\cQ$: an equivariant bijection of $Y$ permutes $\cC$, so the constants have distinct images and the map lies outside $\cM$ by Theorem~\ref{thm:main}, and its rank is $4+8=12$. Hence
\[
|\cQ|=65\,536-32\cdot16=65\,024.
\]
Finally, the map in Example~\ref{ex:missed} has image $O_1\cup P_a\cup P_b\cup P_{ab}\cup\{c_0\}$, so its rank is $4+6+1=11$, and it lies outside $\cM$. This proves $\cM\subsetneq\cQ$.
\end{proof}

\section{The lazy-only monoid and exact enumeration}\label{sec:computation}

We describe the formal and computational verifications in the supplementary material~\cite{supp}. The structural results above do not use them. Write $\cD=\langle\cL\rangle$. For a finite generating set $\Gamma$, let
\[
\ell_\Gamma(T)=\min\{k\geq0:T=\gamma_k\cdots\gamma_1,\ \gamma_i\in\Gamma\}
\]
for $T\in\langle\Gamma\rangle$; the empty product is $\id$.

\begin{lemma}\label{lem:lazycount}
There are $54$ distinct lazy cellular automata. None is invertible, and no nonempty product of lazy cellular automata is the identity.
\end{lemma}
\begin{proof}
The set of configurations changed at $e$ by $\lambda_{S,p}$ is the cylinder
\[
\{x\in X:x|_S=p\}.
\]
This cylinder determines $S$ and $p$: its fixed coordinates are exactly those in $S$. Hence distinct pairs $(S,p)$ give distinct maps, and
\[
|\cL|=\sum_{e\in S\subseteq G}2^{|S|}=2(1+2)^3=54.
\]
For a lazy map, let $i=1-p(e)$ be its writing symbol. Under iteration, the set of coordinates with value $i$ can only increase. If the map were a permutation of the finite set $X$, every configuration would lie on a cycle, forcing this set to be unchanged at each step and every configuration to be fixed. However, any extension of $p$ changes at $e$. Thus each lazy map is noninvertible. A product of transformations of a finite set is invertible only if every factor is invertible, proving the final assertion.
\end{proof}

\begin{proposition}\label{prop:noswaps}
Every $T\in\cD$ satisfies
\begin{equation}\label{eq:noswaps}
T(\cC)=\cC\ \Longrightarrow\ T|_{\cC}=\id_{\cC},\qquad
T(P_s)=P_s\ \Longrightarrow\ T|_{P_s}=\id_{P_s}\quad(s\ne e).
\end{equation}
Exactly $19\,360$ elements of $\cM$ satisfy these conditions.
\end{proposition}
\begin{proof}
Write $T=\lambda_k\cdots\lambda_1$, with $\lambda_j\in\cL$. The assertion is immediate for $k=0$. If $T(\cC)=\cC$, then $2=|\cC|\geq|\lambda_1(\cC)|\geq\cdots\geq|T(\cC)|=2$, so every intermediate image of $\cC$ has two elements and therefore equals $\cC$. Each $\lambda_j|_{\cC}$ is consequently a permutation. A lazy map fixes the constant with its writing symbol, so this permutation is the identity.

Suppose $T(P_s)=P_s$. By~\eqref{eq:lower}, each intermediate image is $P_s$ or a singleton in $\cC$. Since constants cannot return to $P_s$, every intermediate image equals $P_s$. A lazy map changes bits in only one direction, so $w(\lambda_j(x))=w(x)$ implies $\lambda_j(x)=x$. Both weights are two for $x\in P_s$, proving the second implication.

A map satisfying~\eqref{eq:noswaps} has three choices on $\cC$ (the identity or either constant map), three on each $P_s$ (the identity or either constant map), and $16$ on each free-orbit representative. Among these $3^4\cdot16^2$ maps, Theorem~\ref{thm:main} excludes $32\cdot3^3$ with distinct constant images and $32\cdot2\cdot2^3$ with equal constant images but no two-point witness. Hence the required count is
\[
3^4\cdot16^2-32(3^3+2\cdot2^3)=19\,360.
\]
\end{proof}

\begin{proposition}[Exact computation]\label{prop:computed}
We have
\[
|\cD|=4\,007,\qquad
\max_{T\in\cD}\ell_{\cL}(T)=
\max_{T\in\cM}\ell_{\cU\cup\cL}(T)=6.
\]
The word-length distributions are given in Table~\ref{tab:depth}. The semigroup of nonempty lazy compositions has $4\,006$ elements.
\end{proposition}

\begin{table}[ht]
\centering
\caption{Numbers of elements at each shortest-word length.}\label{tab:depth}
\begin{tabular}{@{}crr@{}}
\toprule
Length $k$ & $\ell_{\cL}(T)=k$ & $\ell_{\cU\cup\cL}(T)=k$\\
\midrule
0 & 1 & 1\\
1 & 54 & 565\\
2 & 678 & 8\,806\\
3 & 1\,894 & 27\,424\\
4 & 1\,218 & 19\,004\\
5 & 156 & 3\,870\\
6 & 6 & 42\\
\midrule
Total & 4\,007 & 59\,712\\
\bottomrule
\end{tabular}
\end{table}

\begin{proof}[Verification of the computation]
Number the cells $(e,a,b,ab)$ by $0,1,2,3$, so that group multiplication is bitwise exclusive-or. Encode $x\in X$ by
\[
\iota(x)=\sum_{g=0}^3x(g)2^g.
\]
Encode a local rule by an integer $r\in\{0,\ldots,65\,535\}$ whose bit in position $\iota(x)$ is $T_r(x)(e)$. Thus
\begin{equation}\label{eq:encoding}
T_r(x)(g)=\operatorname{bit}_{\iota(g\cdot x)}(r).
\end{equation}
All cellular automata occur exactly once. The identity has rule number $43\,690$.

The C++ program \nolinkurl{compute_klein_closures.cpp} performs a breadth-first search from the identity along the edges $T\to\gamma T$, first for $\Gamma=\cL$ and then for $\Gamma=\cU\cup\cL$, and replays every word it emits. For each rule it records membership, a shortest length, a parent link, and a generator word in application order, in the files \nolinkurl{lazy-membership.csv} and \nolinkurl{units-lazy-membership.csv}. The units are exactly the rules whose global maps have $16$ distinct images, and the lazy rules, listed in \nolinkurl{lazy-generators.csv}, are obtained by enumerating all pairs $(S,p)$ from Lemma~\ref{lem:lazycount}.

The Python checker \nolinkurl{verify_klein_closures.py} independently regenerates the maps and generators from~\eqref{eq:encoding}; it neither imports the search code nor uses its output summary as evidence. For each reported member set $H$ and distance function $d$, it verifies the following:
\begin{enumerate}
\item $\id\in H$, and every $T\in H$ has a certified generator word of length $d(T)$;
\item $\gamma T\in H$ for all $T\in H$ and $\gamma\in\Gamma$;
\item $d(\gamma T)\leq d(T)+1$ for every such edge, with $d(\id)=0$.
\end{enumerate}
The first check gives $H\subseteq\langle\Gamma\rangle$, and the second gives the reverse inclusion. The words give $\ell_\Gamma(T)\leq d(T)$, whereas the edge inequalities along any word give $d(T)\leq\ell_\Gamma(T)$. Thus the certificates establish both exact membership and shortest lengths. Counting the verified rows gives Table~\ref{tab:depth}. Lemma~\ref{lem:lazycount} shows that deleting the identity gives precisely the nonempty lazy compositions.
\end{proof}

The supplementary material also contains a Lean~4~\cite{lean4} formalization, built on Mathlib~\cite{mathlib} and pinned to Lean v4.33.0, of the criterion of Theorem~\ref{thm:main} (theorems \nolinkurl{KleinFour.membership_iff} and \nolinkurl{KleinFour.membership_iff_orbits}) and of the cardinalities $|\cM|=59\,712$ and $|\cD|=4\,007$ (theorems \nolinkurl{KleinFour.membership_card} and \nolinkurl{KleinFour.lazy_membership_card}). The proofs contain no \texttt{sorry} and do not use \texttt{native\_decide}; they are checked by Lean's kernel and depend only on the axioms \texttt{propext}, \texttt{Classical.choice} and \texttt{Quot.sound}. The word-length distributions in Table~\ref{tab:depth} are not part of the formalization.

\begin{remark}
The length-one entry for the mixed monoid is $565$, since the $512$ units include the identity and no lazy map is a unit. The factorization in Corollary~\ref{cor:generators} uses a smaller generating family; it does not assert the same shortest-word lengths for that family.
\end{remark}

The supplementary material is archived on Zenodo~\cite{supp} and mirrored at \url{https://github.com/the-omega-institute/klein-four-lazy-ca-supplement}. Its directory \nolinkurl{experiments} contains the two programs above, the certificate files, the file \nolinkurl{table-1.csv}, checksums in \nolinkurl{SHA256SUMS}, and the commands that rerun both searches and the checker. The Lean project builds with \nolinkurl{lake build}, and \nolinkurl{Audit.lean} prints the axioms on which the main theorems depend. All computations use exact integer operations. A necessary-and-sufficient structural criterion for membership in $\cD$ remains open.

%%%%%%%%%%%%%%%%%%%%%%%%%%%%%%%%%%%%

\section*{Acknowledgments} 
E.~Alcal\'{a}-Arroyo was supported by SECIHTI \textit{Becas Nacionales para Estudios de Posgrado}, Government of Mexico. 

%%%%%%%%%%%%%%%%%%%%%%
\section*{Declaration on the use of artificial intelligence}

The authors used several models of the GPT (OpenAI) and Claude (Anthropic) families for mathematical discovery, drafting, the Lean formalization and verification scripts, and independent checking of the results. The authors reviewed and verified all results and take full responsibility for the content of this article.


\begin{thebibliography}{9}
\bibitem{order}
E. Alcal\'{a}-Arroyo and A. Castillo-Ramirez,
\emph{On the order of lazy cellular automata},
Theoret. Comput. Sci. \textbf{1076} (2026), 115965.
\href{https://doi.org/10.1016/j.tcs.2026.115965}{doi:10.1016/j.tcs.2026.115965}.

\bibitem{generation}
E. Alcal\'{a}-Arroyo and A. Castillo-Ramirez,
\emph{When do lazy and invertible cellular automata generate all cellular automata?}, October 2026. \href{
https://doi.org/10.48550/arXiv.2610.18480}{
https://doi.org/10.48550/arXiv.2610.18480}.


\bibitem{supp}
E. Alcal\'{a}-Arroyo, A. Castillo-Ramirez, H. Ma, and W. Zhang,
\emph{Supplementary material for ``Generation by lazy and invertible cellular automata over the Klein four-group''},
Zenodo, 2026.
\href{https://doi.org/10.5281/zenodo.23264124}{doi:10.5281/zenodo.23264124}.

\bibitem{finite}
A. Castillo-Ramirez and M. Gadouleau,
\emph{Cellular automata and finite groups},
Nat. Comput. \textbf{18} (2019), no.~3, 445--458.
\href{https://doi.org/10.1007/s11047-017-9640-3}{doi:10.1007/s11047-017-9640-3}.

\bibitem{rank}
A. Castillo-Ramirez and R. H. Ruiz-Medina,
\emph{The relative rank of the endomorphism monoid of a finite $G$-set},
Semigroup Forum \textbf{106} (2023), 51--66.
\href{https://doi.org/10.1007/s00233-023-10340-7}{doi:10.1007/s00233-023-10340-7}.

\bibitem{mathlib}
The mathlib Community,
\emph{The Lean mathematical library},
in: Proceedings of the 9th ACM SIGPLAN International Conference on Certified Programs and Proofs (CPP 2020), ACM, 2020.
\href{https://doi.org/10.1145/3372885.3373824}{doi:10.1145/3372885.3373824}.

\bibitem{lean4}
L. de Moura and S. Ullrich,
\emph{The Lean~4 theorem prover and programming language},
in: Automated Deduction -- CADE 28 (A. Platzer and G. Sutcliffe, eds.), Lecture Notes in Comput. Sci., vol. 12699, Springer, 2021.
\href{https://doi.org/10.1007/978-3-030-79876-5_37}{doi:10.1007/978-3-030-79876-5\_37}.
\end{thebibliography}
\end{document}